\documentclass[11pt]{article}

\usepackage[all]{xy} 
\usepackage{amsmath,amssymb,amsthm}
\usepackage{amsrefs}
\usepackage{bm} 
\usepackage{bbm} 
\usepackage{color} 
\usepackage{comment} 
\usepackage{bigints} 
\usepackage{hyperref}
\usepackage{authblk}
\usepackage{fullpage}

\newtheorem{theorem}{Theorem}[section]
\newtheorem{proposition}[theorem]{Proposition}

\newtheorem{lemma}[theorem]{Lemma}

\newtheorem{remark}[theorem]{Remark}
\newtheorem{definition}[theorem]{Definition}

\newtheorem*{question*}{Question}

\newcommand{\supp}{\mathrm{supp}}

\allowdisplaybreaks

\title{Differentiation properties of class ${L}^1([0,1]^2)$ with respect to two different basis of rectangles}

\begin{document}

\date{} 

\author{Michihiro Hirayama\thanks{hirayama@math.tsukuba.ac.jp}}
\affil[1]{Department of Mathematics, University of Tsukuba, Tennodai, Ibaraki 305-8571, Japan}

\author{Davit Karagulyan\thanks{davitk@kth.se}}
\affil[2]{KTH Royal Institute of Technology, SE-100 44 Stockholm, Sweden}

\maketitle

\begin{abstract}
It is a well known result by Saks \cite{Saks1934} that there exist a function $f \in L^1(\mathbb{R}^2)$ so that for almost every $(x,y)\in \mathbb{R}^2$
\[
\lim_{\substack{\mathrm{diam} R\rightarrow 0, \\ (x,y) \in R \in \mathcal{R}}}\left|\frac{1}{|R|}\int_R f(x,y)\, dxdy\right|=\infty,
\]
where $\mathcal{R}=\{[a,b)\times [c,d)\colon a<b, c<d\}$.
In this note we address the following question: assume we have two different collections of rectangles; under which conditions there exists a function $f \in L^1(\mathbb{R}^2)$ so that its integral averages are divergence with respect to one collection and convergence with respect to another?
More specifically, let $\mathcal{D}, \mathcal{C} \subset (0,1]$ and consider rectangles with side lengths in $\mathcal{D}$ and respectively in $\mathcal{C}$. 
We show that if the sets $\mathcal{D}$ and $\mathcal{C}$ are sufficiently ``far" from each other, then such a function can be constructed. 
We also show that in the class of positive functions our condition is also necessary for such a function to exist.
\end{abstract}
\section{Introduction}

Let $\mathcal{R}=\{[a,b)\times [c,d)\}$ be the set of all rectangles with their sides parallel to the coordinate axis. 
Given a collection $\mathcal{C} \subset (0,1]$, let $\mathcal{R}_\mathcal{C}\subset \mathcal{R}$ be the collection of all rectangles $[a,b)\times [c,d)$ so that $b-a \in \mathcal{C}$ and $d-c \in \mathcal{C}$. 
 
\begin{definition} 
A family of rectangles $\mathcal{M} \subset \mathcal{R}$ is said to be a basis of differentiation (or simply a basis), if for any point $z \in \mathbb{R}^{2}$ there exists a sequence of rectangles $R_{k} \in \mathcal{M}$ such that $z \in R_{k}$, $k\in \mathbb{N}$, and $\mathrm{diam} R_{k}\rightarrow 0$ as $k \rightarrow \infty$.
\end{definition}

Let $\mathcal{C} ,\mathcal{D}\subset (0,1]$ be two collections. 
Thus $\mathcal{R}_\mathcal{C}$ and $\mathcal{R}_\mathcal{D}$ will be basis of differentiation if and only if $\liminf \mathcal{C}=0$ and $\liminf \mathcal{D}=0$. Let $\mathcal{M} \subset \mathcal{R}$ be a differentiation basis. 
For any function $f \in L^{1}\left(\mathbb{R}^{2}\right)$ we define
\[
\delta _{\mathcal{M}}(z, f)=\limsup _{\substack{\mathrm{diam} R \rightarrow 0, \\ z \in R \in \mathcal{M}}}\left|\frac{1}{|R|} \int_{R} f\, dm-f(z)\right| .
\]
Here and below, we denote by $m$ or $|\cdot |$ the Lebesgue measure on $\mathbb{R} ^2$. 
The function $f \in L^{1}\left(\mathbb{R}^{2}\right)$ is said to be differentiable at a point $z \in \mathbb{R}^{2}$ with respect to the basis $\mathcal{M}$, if $\delta_{\mathcal{M}}(z, f)=0$. 
Denote
\[
\mathcal{F}(\mathcal{M})=\left\{ f \in L^{1}\left(\mathbb{R}^{2}\right) \colon \delta_{\mathcal{M}}(z, f)=0, m\text{-} a.e.\ z \in \mathbb{R}^{2}\right\} .
\]

Let $\Phi \colon \mathbb{R}^{+} \rightarrow \mathbb{R}^{+}$ be a convex function. 
Denote by $\Phi (L)\left(\mathbb{R}^{2}\right)$ the class of measurable functions $f$ defined on $\mathbb{R}^{2}$ such that $\Phi(|f|) \in L^{1}\left(\mathbb{R}^{2}\right) $. 
If $\Phi$ satisfies the $\Delta_{2}$-condition $\Phi(2 x) \leq k \Phi(x)$, then $\Phi (L)\left(\mathbb{R}^{2}\right)$ turns to be an Orlicz space with the norm
\[
\|f\|_{\Phi}=\inf \left\{ c>0\colon \int_{\mathbb{R}^{2}} \Phi\left(\frac{|f|}{c}\right) \, dm\leq 1\right\} .
\]
The following classical theorems of Jessen, Marcinkiewicz, and Zygmund \cite{Jessen-Marcinkiewicz-Zygmund1935}, and Saks \cite{Saks1934} determine the optimal Orlicz space, which functions have a.e. differentiable integrals with respect to the entire family of rectangles $\mathcal{R}$ is the space
\[
L(1+\log L)\left(\mathbb{R}^{2}\right) \subset L^{1}\left(\mathbb{R}^{2}\right)
\]
corresponding to the case $\Phi(t)=t\left(1+\log ^{+} t\right)$. 
See also \cite{deGuzman}. 

\begin{theorem}[Jessen-Marcinkiewicz-Zygmund \cite{Jessen-Marcinkiewicz-Zygmund1935}] \label{J-M-Z}
If
\[
f\in L(1+\log L)\left(\mathbb{R}^{2}\right),
\]
then
\[
\delta _{\mathcal{R}}(z,f)=0
\]
for $m$-almost every $z\in \mathbb{R}^{2}$. 
\end{theorem}

\begin{theorem}[Saks \cite{Saks1934}] 
If
\[
\Phi(t)=o(t \log t) \text { as } t \rightarrow \infty ,
\]
then $\Phi(L)\left(\mathbb{R}^{2}\right) \not \subset \mathcal{F}(\mathcal{R})$.  
Moreover, there exists a positive function $f \in \Phi(L)\left(\mathbb{R}^{2}\right)$ such that $\delta_{\mathcal{R}}(z, f)=\infty$ everywhere.
\end{theorem}

In this paper, we are interested in differentiability property of class $L^1([0,1]^2)$ with respect to 
two basis $\mathcal{R}_\mathcal{C}$ and $\mathcal{R}_\mathcal{D}$. 
We investigate conditions under which there exists a function $f \in L^1([0,1]^2)$, so that
\[
\delta_{\mathcal{R}_\mathcal{C}}(z,f)=0\ \text{ for $m$-almost every } z\in [0,1]^2
\]
and
\[
\delta_{\mathcal{R}_\mathcal{D}}(z,f)\neq 0\ \text{ for $m$-almost every } z\in [0,1]^2.
\]

To the best of the authors knowledge there is only one result that has some relation to the problem considered in this note.
In \cite{Karagulyan-Karagulyan-Safaryan2017} the authors study equivalence of differentiation basis of dyadic rectangles. More specifically consider the basis
\begin{equation}\label{dyd-bas}
\mathcal{R}^{\text {dyadic }}=\left\{\left[\frac{i-1}{2^{n}}, \frac{i}{2^{n}}\right) \times\left[\frac{j-1}{2^{m}}, \frac{j}{2^{m}}\right) \colon i, j, n, m \in \mathbb{Z}\right\}.
\end{equation}
Let $\Delta=\left\{\nu_{k}: k=1,2, \ldots\right\}$ be an increasing sequence of positive integers. 
This sequence generates the rare basis $\mathcal{R}_{\Delta}^{\text {dyadic }}$ of dyadic rectangles of the form \eqref{dyd-bas} with $n, m \in \Delta$. 
This kind of bases have also been considered in several papers \cites{Hagelstein2000, Hare-Stokolos2000, Stokolos2006, Karagulyan2013}. 
In \cite{Karagulyan-Karagulyan-Safaryan2017} the authors study under which conditions the basis 
$\mathcal{R}^{\text {dyadic }}$ will be equivalent to the basis given by $\mathcal{R}^{\text {dyadic }}_\Delta$. 
We remark that although in \cite{Karagulyan-Karagulyan-Safaryan2017} one only considers dyadic rectangles, this case can be compared with $\mathcal{D}=[0,1]\setminus \mathcal{C}$ in our case. 

Unlike the results considered above, in this note we consider the problem in full generality, namely we and are interested in all rectangles with their sides in $\mathcal{C}$ or $\mathcal{D}$ that contain the point of differentiation. 

The paper is self contained and uses some methods from analysis and probability theory. 
\section{Main results}

Intuitively, in order to have divergence and convergence phenomena simultaneously, the two set $\mathcal{C}$ and $\mathcal{D}$ have to be far from each other. We now formalize this intuition.
For each $x \in \mathcal{D}$, we define the following two numbers
\begin{align*}
\overline{x}&=\sup \left\{ a\in \mathcal{C}\colon a<x\right\} , \\
\underline{x}&=\inf \left\{ a\in \mathcal{C}\colon a>x\right\}.
\end{align*}
The ratios $\overline{x}/x$ and $x/\underline{x}$ denote that distance of $x$ from the set $\mathcal{C}$ from below and above. 
We prove the following theorem. 

\begin{theorem} \label{main-1}
Let $\mathcal{C}, \mathcal{D}\subset (0,1]$ and assume 
\[
\liminf_{x \rightarrow 0, x \in \mathcal{D}}\left( \max \left\{ \frac{\overline{x}}{x},\frac{x}{\underline{x}}\right\} \right) =0.
\]
Then there exists a function $f \in {L}^1([0,1]^2,m)$ so that for $m$-almost all $z \in [0,1]^2$
\[
\delta _{\mathcal{R}_\mathcal{C}}(z, f)=0,
\]
and
\[
\delta _{\mathcal{R}_\mathcal{D}}(z, f)=\infty.
\]
\end{theorem}

The function that will be constructed in Theorem \ref{main-1} is unbounded both from above and below. This leads as to the following question

\begin{question*}
Does there exist a positive function $f$ satisfying the conditions of Theorem \ref{main-1}?
\end{question*}

We have the following theorem in the opposite direction

\begin{theorem} \label{main-2}
Let $\mathcal{C}, \mathcal{D}\subset (0,1]$ and assume that
\[
\liminf_{x \rightarrow 0, x \in \mathcal{D}}\left( \max \left\{ \frac{\overline{x}}{x},\frac{x}{\underline{x}}\right\} \right) > 0.
\]
Let $f \in L^1([0,1]^2,m)$, with $f \geq 0$ almost surely. 
If
\[
\delta _{\mathcal{R}_\mathcal{C}}(z, f)=0\ \text{ for $m$-almost every } z\in [0,1]^2, 
\]
then 
\[
\delta _{\mathcal{R}_\mathcal{D}}(z, f)=0\ \text{ for $m$-almost every } z\in [0,1]^2.
\]
\end{theorem}

The proof of Theorem \ref{main-1} is based on the following theorem.

\begin{theorem} \label{main-3}
For every $\varepsilon, \delta>0$ and $n>0$, there exist a function $f\in L^\infty ([0,1]^2,m)$, $\eta =\eta (\varepsilon ,\delta ,n)\in (0,\delta )$ and a set $Q\subset [0,1]^2$, so that
\begin{equation} \label{c-1}
\| f\| _{L^1(m)}\leq 4+o(1),
\end{equation}
and $|Q|>1-\varepsilon$ such that
\begin{enumerate}
\item[{\rm (i)}] If $z \in Q$ and $z \in A \in \mathcal{R}_\mathcal{C}$, then
\begin{equation*} \label{c-2}
\left| \frac{1}{|A|}\int_A f\, dm\right| \leq 2.	
\end{equation*}
\item[{\rm (ii)}] If $z \in Q$, then there exists $B \in \mathcal{R}_\mathcal{D}$, with $z \in B$ and $|B|>\eta $, so that
\begin{equation*} \label{c-3}
\left| \frac{1}{|B|} \int_B f\, dm\right| >n.
\end{equation*}
\item[{\rm (iii)}]
 For every $R \in \mathcal{R}$ with $|R|>\delta$, we have
\begin{equation*} \label{c-4}
\left| \frac{1}{|R|}\int_R f\, dm\right| <1.
\end{equation*}
\end{enumerate}
\end{theorem}
\section{Preliminaries}

\subsection{Auxiliary construction}

Let $n\in \mathbb{N} $. 
Given a decreasing sequence $b_1>b_2>\dots >b_n>\dots >b_{2n}$, let 
\[
B_1=[0,b_1]\times [0,b_{2n}], B_2=[0,b_2]\times [0,b_{2n-1}], \dots ,B_n=[0,b_n]\times [0,b_{n+1}].
\]
Here the first factor is the height and the second one is the width. 
Suppose that the sequence $(b_j)$ satisfies 
\begin{equation} \label{area_decrease}
b_1b_{2n}\leq b_2b_{2n-1}\leq \dots \leq b_{n-1}b_{n+2}\leq b_nb_{n+1}.
\end{equation}
In other words, we suppose that 
\[
|B_1|\leq |B_2|\leq \dots \leq |B_{n-1}|\leq |B_n|.
\]
For $j=1,\dots ,n-1$, let $q_j\in \mathbb{N} $ be such that $q_j=\lfloor |B_{j+1}|/|B_{j}|\rfloor $. 

Let $\Theta $ be the set of $(\theta _{n-1},\theta _{n-2},\dots ,\theta _2,\theta _1)$ with $\theta _j\in \{ 0,1,\dots ,q_j\} $ satisfying the following condition: if $\theta _k=0$ for some $k$, then $\theta _l=0$ for every $l\in \{ k,k-1,\dots ,1\} $. 
For each $\theta=(\theta _{n-1},\theta _{n-2},\dots ,\theta _1)\in \Theta $, let 
\[
|\theta|=\begin{cases} 1,& \text{if all } \theta _j\neq 0\\ \max \{1\leq j\leq n-1\colon \theta _j=0\} +1, & \text{otherwise} \end{cases}
\]
and define rectangles $B(\theta )$ as follows. 
For $\theta\in \Theta $ with $|\theta |=n$, that is $\theta _{n-1}=0$, we let
\[
B(\theta )=\left[ 0,b_{n}\right] \times \left[ 0,b_{n+1}\right] =B_n.
\]
For $\theta\in \Theta $ with $|\theta |\in \{ 1,\dots ,n-1\} $, we let 
\[
B(\theta )=\left[ 0,b_{|\theta| }\right] \times \left[ \sum _{j=|\theta |}^{n-1}(\theta _j-1) b_{2n+1-j},\sum _{j=|\theta |}^{n-1}(\theta _j-1) b_{2n+1-j} +b_{2n+1-|\theta|}\right] ,
\]
where note that $\theta _{n-1}\in \{ 1,\dots ,q_{n-1}\} $ as $|\theta |\neq n$.
Note also that for each $\theta \in \Theta $, one has $|B(\theta)|=b_{|\theta|}b_{2n+1-|\theta|}$. 

Define subsets of $[0,1]^2$ by 
\[
E=\bigcup _{\theta \in \Theta }B(\theta )\quad \text{and} \quad F=\bigcup _{\substack{\theta \in \Theta \colon \\ |\theta |<n}}B(\theta ).
\]
One has
\begin{align}
E&=\bigcup _{j=1}^n\bigcup _{|\theta |=j}B(\theta )=\left( \bigcup _{j=1}^{n-1}\bigcup _{|\theta |=j}B(\theta )\right) \cup B_n, \nonumber \\
F&=\bigcup _{j=1}^{n-1}\bigcup _{|\theta |=j}B(\theta ), \label{def_F}
\end{align}
where for every $j\in \{ 1,\dots ,n-1\} $  there are $q_{n-1}\cdots q_j$-many rectangles $B(\theta )$ of $|\theta |=j$. 
\subsection{Area estimates}

\begin{lemma} \label{area}
Suppose that there is $\lambda \in (0,1)$ such that $b_{k+1}/b_k<\lambda $ for every $k=1,\dots ,2n-1$. 
Then one has 
\begin{equation} \label{areaE}
\frac{1}{n} +\frac{1- \lambda }{n} \sum _{j=1}^{n-1}\frac{q_{n-1}\cdots q_{j}}{(q_{n-1}+1)\cdots (q_{j}+1)} \leq \frac{|E|}{n|B_n|} \leq 1,
\end{equation}
and
\begin{equation} \label{areaF}
\frac{1- \lambda }{n-1} \sum _{j=1}^{n-1}\frac{q_{n-1}\cdots q_{j}}{(q_{n-1}+1)\cdots (q_{j}+1)} \leq \frac{|F|}{(n-1)|B_n|} \leq 1.
\end{equation}
\end{lemma}
\begin{proof}
By construction, one has
\begin{align*} 
|E|&=b_nb_{n+1}+q_{n-1}b_{n+2}(b_{n-1}-b_{n})+q_{n-1}q_{n-2}b_{n+3}(b_{n-2}-b_{n-1})  \\
&\quad +\cdots +q_{n-1}\cdots q_{1}b_{2n}(b_{1}-b_{2})  \\
&=|B_n| +q_{n-1}|B_{n-1}|+q_{n-1}q_{n-2}|B_{n-2}|+\cdots +q_{n-1}\cdots q_{1}|B_{1}| \\
&\quad -(q_{n-1}b_{n+2}b_{n}+q_{n-1}q_{n-2}b_{n+3}b_{n-1}+\cdots +q_{n-1}\cdots q_{1}b_{2n}b_{2}) .
\end{align*}
Since 
\[
q_{n-1}\cdots q_{j+1}q_j|B_j|\leq q_{n-1}\cdots q_{j+1}|B_{j+1}|\leq \dots \leq q_{n-1}|B_{n-1}|\leq |B_n|
\]
for every $j=1,\dots ,n-1$, it follows that $|E|\leq n|B_n|$. 

On the other hand, one has
\begin{align*}
\frac{|E|}{n|B_n|} 
&=\frac{|B_n|+\sum _{j=1}^{n-1}q_{n-1}\cdots q_{j}b_{2n+1-j}(b_{j}-b_{j+1})}{n|B_n|} \\
&=\frac{1}{n} +\frac{1}{n} \sum _{j=1}^{n-1}\frac{q_{n-1}\cdots q_{j}b_{2n+1-j}(b_{j}-b_{j+1})}{|B_n|} .
\end{align*}
Since 
\[
|B_n|\leq (q_{n-1}+1)|B_{n-1}|\leq \dots \leq (q_{n-1}+1)\cdots (q_{j}+1)|B_j|
\]
for each $j=1,\dots ,n-1$, it follows that
\begin{align*}
\frac{q_{n-1}\cdots q_{j}b_{2n+1-j}(b_{j}-b_{j+1})}{|B_n|} 
&\geq \frac{q_{n-1}\cdots q_{j}}{(q_{n-1}+1)\cdots (q_{j}+1)} \frac{b_{2n+1-j}(b_{j}-b_{j+1})}{b_jb_{2n+1-j}} \\
&=\frac{q_{n-1}\cdots q_{j}}{(q_{n-1}+1)\cdots (q_{j}+1)} \left( 1- \frac{b_{j}}{b_{j+1}} \right) \\
&\geq \frac{q_{n-1}\cdots q_{j}}{(q_{n-1}+1)\cdots (q_{j}+1)} \left( 1-\lambda \right) ,
\end{align*}
and hence
\begin{align*}
\frac{|E|}{n|B_n|} 
&\geq \frac{1}{n} +\frac{1- \lambda }{n} \sum _{j=1}^{n-1}\frac{q_{n-1}\cdots q_{j}}{(q_{n-1}+1)\cdots (q_{j}+1)}.
\end{align*}
We have obtained \eqref{areaE}. 
Since the proof of \eqref{areaF} is same as that of \eqref{areaE}, we omit it. 
\end{proof}

Let 
\begin{align*}
K_E&=\left( \frac{1}{n} +\frac{1- \lambda }{n} \sum _{j=1}^{n-1}\frac{q_{n-1}\cdots q_{j}}{(q_{n-1}+1)\cdots (q_{j}+1)} \right) ^{-1}, \\
K_F&=\left( \frac{1- \lambda }{n-1} \sum _{j=1}^{n-1}\frac{q_{n-1}\cdots q_{j}}{(q_{n-1}+1)\cdots (q_{j}+1)}\right) ^{-1}.
\end{align*}
Then it follows from \eqref{areaE} and \eqref{areaF} in Lemma \ref{area} that
\begin{equation} \label{comparison_E_F}
K_E^{-1}\frac{n}{n-1} \leq \frac{|E|}{|F|} \leq K_F\frac{n}{n-1} .
\end{equation}
Notice that $K_E,K_F\geq 1$ and they can be arbitrarily close to one by taking $\lambda \in (0,1)$ small and $(b_n)$ relevantly. 
\section{Main Lemma}

Without loss of generality we can assume that $\mathcal{D}$ is a sequence, i.e. $\mathcal{D}=\{b_n\}_{n\geq 1}$.
For each $b_n \in \mathcal{D}$, we define
\begin{align*}
\overline{a} _n&=\sup \left\{ a\in \mathcal{C}\colon a<b_n\right\} , \\
\underline{a} _n&=\inf \left\{ a\in \mathcal{C}\colon a>b_n\right\}.
\end{align*}
Recall that we are given $\mathcal{C}, \mathcal{D}\subset (0,1]$, where $\mathcal{D}=\{b_n\}_{n \in \mathbb{N}}$, and assume 
\begin{equation} \label{regularity}
\liminf_{n \rightarrow \infty}\left( \max \left\{ \frac{\bar a_n}{b_n},\frac{b_n}{\underline{a}_n}\right\} \right) =0.
\end{equation}
Throughout this section, we assume \eqref{regularity}. 
In this section, we prove the following lemma which will play a fundamental role for the proof of Theorems. 

\begin{lemma} \label{main_lemma} 
Given $\varepsilon ,\delta >0$ and $n\in \mathbb{N}$, there exist sets $F, D \subset [0,1]^2$, a function $h\in L^1([0,1]^2,m)$, and $\eta =\eta (\varepsilon ,\delta ,n)\in (0,\delta )$ satisfying $|D|\leq \varepsilon |F|$ and  
\begin{equation} \label{order_ratio_L1norm_areaF}
\frac{\| h\| _{L^1(m)}}{|F|}=2+o(1), 
\end{equation}
such that for every $z \in [0,1]^2$ we have the following. 
\begin{enumerate}
\item[{\rm (i)}] If $z \notin D$ and $z \in A \in \mathcal{R}_\mathcal{C}$, then
\[
\left| \frac{1}{|A|}\int_A h\, dm\right| \leq 1.
\]
\item[{\rm (ii)}] If $z \in F$, then there exists $B \in \mathcal{R}_\mathcal{D}$, with $z \in B$ and $|B|>\eta $, so that
\[
\left| \frac{1}{|B|}\int_B h\, dm\right| >n. 
\]
\item[{\rm (iii)}] For every $R \in \mathcal{R}$ with $|R|>\delta $, we have
\[
\left| \frac{1}{|R|}\int_R h\, dm\right| <1/2.
\]
\end{enumerate}
\end{lemma}
\subsection{Proof of Lemma \ref{main_lemma}} 

We start with construction of a function and sets in the following two sections. 
Then we complete the proof of Lemma \ref{main_lemma} in Section \ref{proof_main_lemma}. 
\subsubsection{Construction of the function} \label{function}

Define $\Theta ^*=\{ \theta \in \Theta \colon |\theta |=1\} $. 
For $\tau \in (0,b_{2n})$ and $\theta \in \Theta ^*$, define a subset of $B(\theta )$ by
\[
B_\tau (\theta )=[0,\tau ]\times \left[ \sum _{j=1}^{n-1}(\theta _j-1) b_{2n+1-j},\sum _{j=1}^{n-1}(\theta _j-1) b_{2n+1-j} +\tau \right] .
\]
Namely $B_\tau (\theta )$ is the square with side lengths $\tau $ at the bottom left corner of $B(\theta)$ for $\theta \in \Theta ^*$. 
Define also $\sigma \colon \Theta ^*\to \{ 1,-1\} $ by 
\[
\sigma (\theta )=\begin{cases} +1, \text{ if } \theta _{n-1} \text{ is odd},\\-1, \text{ if } \theta _{n-1} \text{ is even}. \end{cases}
\]
Then define $h\colon [0,1]^2\to \mathbb{R}$ by
\begin{equation} \label{def_h}
h(z)=\begin{cases} 
\sigma (\theta )\dfrac{n|B_n|}{\tau ^2q_1\cdots q_{n-1}}, & \text{if } z \in B_\tau (\theta ) \text{ for some } \theta \in \Theta ^*, \\
0, & \text{otherwise}. 
\end{cases} 
\end{equation}

\begin{lemma} \label{average_over_B}
For every $\theta \in \Theta $, one has
\[
n\leq \frac{1}{|B(\theta )|} \left| \int _{B(\theta )}h\, dm\right| \leq n\prod _{j=|\theta |}^{n-1}\left( 1+\frac{1}{q_j}\right). 
\]
\end{lemma}
\begin{proof}
By construction, one has 
\begin{align*}
\left| \int _{B(\theta )}h\, dm\right| 
&=\frac{n|B_n|}{\tau ^2q_1\cdots q_{n-1}}\times |B(\theta )\cap \text{supp} h| \\
&=\frac{n|B_n|}{\tau ^2q_1\cdots q_{n-1}}\times \begin{cases} \tau ^2q_1\cdots q_{|\theta |-1}, & |\theta |>1 \\ \tau ^2, & |\theta |=1 \end{cases} \\
&=\frac{n|B_n|}{q_{|\theta |}\cdots q_{n-1}},
\end{align*}
and hence
\[
\frac{1}{|B(\theta )|} \left| \int _{B(\theta )}h\, dm\right| =\frac{1}{|B(\theta )|} \frac{n|B_n|}{q_{|\theta |}\cdots q_{n-1}}\geq \frac{n|B_n|}{|B_{n}|}=n.
\]

Since $|B_n|\leq (q_{n-1}+1)|B_{n-1}|\leq \cdots \leq (q_{n-1}+1)\cdots (q_{|\theta |}+1)|B(\theta )|$, one also has
\[
\frac{1}{|B(\theta )|} \left| \int _{B(\theta )}h\, dm\right| =\frac{1}{|B(\theta )|} \frac{n|B_n|}{q_{|\theta |}\cdots q_{n-1}}\leq n\prod _{j=|\theta |}^{n-1}\left( 1+\frac{1}{q_j}\right).
\]
\end{proof}

\begin{lemma} \label{ratio_L1norm_area}
Let $n\in \mathbb{N}$. 
Then we have
\begin{equation} \label{ratio_L1norm_areaE}
2\leq \frac{\| h\| _{L^1(m)}}{|E|}\leq 2K_E,
\end{equation}
and
\begin{equation} \label{ratio_L1norm_areaF}
2K_E^{-1}\frac{n}{n-1}\leq \frac{\| h\| _{L^1(m)}}{|F|}\leq 2K_EK_F\frac{n}{n-1}.
\end{equation}
\end{lemma}
\begin{proof}
Note first that 
\[
\| h\| _{L^1([0,1]^2,m)}=2\int _F|h|\, dm.
\] 
It follows from the construction that 
\begin{align*}
\int _F|h|\, dm=\frac{n|B_n|}{\tau ^2q_1\cdots q_{n-1}}\times |F\cap \text{supp} h| =\frac{n|B_n|}{\tau ^2q_1\cdots q_{n-1}}\times \tau ^2q_1\cdots q_{n-1} =n|B_n|,
\end{align*}
and hence Lemma \ref{area}-\eqref{areaE} implies \eqref{ratio_L1norm_areaE}. 
By using \eqref{comparison_E_F}, we have \eqref{ratio_L1norm_areaF} from \eqref{ratio_L1norm_areaE}. 
\end{proof}
\subsubsection{Construction of the exceptional set $D$}

Let $h$ be the function of the form \eqref{def_h} defined in Section \ref{function}. 
In this section, we will construct a set $D$ while examining the integral averages of $h$ with respect to $\mathcal{R}_\mathcal{C}$ by making use of the condition \eqref{regularity}. 
Indeed, one uses the following property on the two sets $\mathcal{C}$ and $\mathcal{D} =\{ b_j\} _{j\in \mathbb{N}}$. 
The proof is a direct consequence of the condition \eqref{regularity}, and is omitted.

\begin{lemma} \label{regular_distortion}
Assume the condition \eqref{regularity}.  
Given $b_{1}>\dots >b_{n}$, and $\lambda _k\in (0,1)$ for $k=1,\dots ,n$, one can choose $b_{n+1}>\dots >b_{2n}$ such that 
\begin{equation*} 
\frac{\overline{a} _{n+k}}{b_{n+k}} \leq \lambda _k\frac{b_{n-k+1}}{\overline{a} _{n-k}}, 
\end{equation*}
and
\begin{equation*} 
\frac{b_{n+k}}{\underline{a} _{n+k}} \leq \lambda _k\frac{\overline{a} _{n-k+1}}{b_{n-k+1}}
\end{equation*}
for $k=1,\dots ,n$, where $\overline{a} _{0}=1$ as a convention. 
\end{lemma}

Henceforth, we denote rectangles with side lengths $x$ and $y$ by $A_{xy}$, where $x$ is the length of vertical side and $y$ is that of the horizontal one. 
In other words, we assume that $x$ is the height and $y$ is the width of our rectangle. 
Let 
\[
\mathcal{A} _\mathcal{C} =\left\{ A_{xy}\subset [0,1]^2\colon A_{xy}\cap \supp h\neq \emptyset ,\ x,y\in \mathcal{C} \right\} .
\]

We primarily divide our argument into the following cases with respect to the heights of $A_{xy}\in \mathcal{A} _\mathcal{C} $: 1) $x\leq b_{n+1}$, 2) $x\in (b_{n},b_{1}]$, 3) $x>b_{1}$, and 4) $x\in (b_{n+1},b_{n}]$. 
\begin{description}
\item[Case 1) $x\leq  b_{n+1}$:] 
By assumption \eqref{regularity}, take $b_{n+1}$ so that 
\[
\frac{\overline{a} _{n+1}}{b_{n+1}} \ll b_{n}. 
\]
(Here and below, we will sometimes write $X\ll Y$ if for a given $\lambda \in (0,1)$ one can take $X$ so small that $X\leq \lambda Y$.) 
Then for every $y\in (0,1]$ one has 
\[
xy\leq x\leq \overline{a} _{n+1}\ll b_nb_{n+1}=|B_{n}|.
\]
Let 
\[
D_1=\bigcup _{A_{xy}\in \mathcal{A} _\mathcal{C}} \left\{ \text{int} A_{xy}\colon (x,y)\in (0,b_{n+1}]\times (0,1] \right\} .
\]
\item[Case 2) $b_{n}<x\leq b_{1}$:]
There is $r=r(x)\in \{ 1,\dots ,n-1\}$ such that
\[
b_{r+1}<x\leq b_{r}.
\]
We consider two cases depending on the side $y$. 
\begin{description}
\item[2-i) The case $y>b_{2n-r}$.] 
We begin with the case where $A=A_{xy}$ contains $B(\theta )\in \mathcal{R} _\mathcal{D}$ with $|\theta |=r+1$, that is height $b_{r+1}$ and width $b_{2n-r}$. 
More precisely, there is $p=p(y)\in \mathbb{N}$ such that $A$ contains $p$-many disjoint copies of $B_{|\theta |}=B_{r+1}$ and does not contain $(p+1)$-many disjoint copies of it. 
We then have by Lemma \ref{average_over_B} that
\[
\int _{A}h\, dm 
\leq |B(\theta )|(p+2)n\prod _{j=r+1}^{n-1}\left( 1+\frac{1}{q_j}\right) .
\]
Since $A$ contains a rectangle of height $x$ and width $pb_{2n-r}$, it follows that
\begin{align}
\frac{1}{|A|} \int _Ah\, dm 
&\leq \frac{(p+2)nb_{r+1}b_{2n-r}}{xpb_{2n-r}} \prod _{j=r+1}^{n-1}\left( 1+\frac{1}{q_j}\right) \nonumber \\
&=n\frac{p+2}{p} \frac{b_{r+1}}{x} \prod _{j=r+1}^{n-1}\left( 1+\frac{1}{q_j}\right) \nonumber \\
&\leq n\frac{p+2}{p} \frac{b_{r+1}}{\underline{a} _{r+1}} \prod _{j=r+1}^{n-1}\left( 1+\frac{1}{q_j}\right) . \label{case2ia} 
\end{align}

Next, suppose that $A$ contains no rectangles $B(\theta )\in \mathcal{R} _\mathcal{D}$ with $|\theta |=r+1$, that is the case of $p=0$. 
One still has 
\begin{align}
\frac{1}{|A|} \int _Ah\, dm \leq \frac{2nb_{r+1}b_{2n-r}}{xy} \prod _{j=r+1}^{n-1}\left( 1+\frac{1}{q_j}\right) \leq 2n\frac{b_{r+1}}{\underline{a} _{r+1}} \frac{b_{2n-r}}{\underline{a} _{2n-r}} \prod _{j=r+1}^{n-1}\left( 1+\frac{1}{q_j}\right) . \label{case2ib} 
\end{align}
\item[2-ii) The case $y\leq b_{2n-r}$.] 
By definition, one has
\[
\frac{y}{b_{2n-r}} \leq \frac{\overline{a} _{2n-r}}{b_{2n-r}} \quad \text{and} \quad \frac{b_{r+1}}{x} \geq \frac{b_{r+1}}{\overline{a} _{r}} .
\]
By applying Lemma \ref{regular_distortion} with $k=n-r$, one obtains
\[
\frac{\overline{a} _{2n-r}}{b_{2n-r}} \leq \lambda \frac{b_{r+1}}{\overline{a} _{r}} ,
\]
and thus
\[
\frac{y}{b_{2n-r}} \leq \frac{\overline{a} _{2n-r}}{b_{2n-r}} \leq \lambda \frac{b_{r+1}}{\overline{a} _{r}}\leq \lambda \frac{b_{r+1}}{x} ,
\]
which means $xy\leq \lambda b_{r+1}b_{2n-r}=|B_{r+1}|$. 
\end{description}
Let 
\[
D_2=\bigcup _{r=1}^{n-1}\bigcup _{A_{xy}\in \mathcal{A} _\mathcal{C}} \left\{ \text{int} A_{xy}\colon (x,y)\in (b_{r+1},b_{r}]\times (b_{2n-r+1},b_{2n-r}] \right\} .
\]
\item[Case 3) $x>b_{1}$:] 
We divide into two cases as follows. 
\begin{description}
\item[3-i) The case $y>b_{2n}$.] 
By the same argument as  Case 2-i), one can show the convergence of integral averages of $h$ over $A_{xy}$. 
\item[3-ii) The case $y\leq b_{2n}$.] 
As in Case 2-ii), one obtains 
\[
\frac{y}{b_{2n}} \leq \frac{\overline{a} _{2n}}{b_{2n}} \leq \lambda \frac{b_{1}}{\overline{a} _{0}}\leq \lambda \frac{b_{1}}{x} 
\]
with making use of Lemma \ref{regular_distortion} with $k=n$, and hence $xy\leq \lambda b_1b_{2n}=\lambda |B_1|$. 
Let 
\[
D_3=\bigcup _{A_{xy}\in \mathcal{A} _\mathcal{C}} \left\{ \text{int} A_{xy}\colon (x,y)\in (b_{1},1]\times (0,b_{2n}] \right\} .
\]
\end{description}
\item[Case 4) $b_{n+1}<x\leq b_{n}$:] 
We divide into two cases as follows. 
\begin{description}
\item[4-i) The case $y>b_{n+1}$.] 
For such an $A=A_{xy}\in \mathcal{A} _\mathcal{C} $, it follows from the definition of $h$ \eqref{def_h} that either $\int _Ah \, dm=0$ or there are at most two $\theta ,\theta '\in \Theta $ with $|\theta |=|\theta '|=n-1$ such that
\[
\left| \int _Ah\, dm\right| \leq \int _A|h| \, dm \leq \int _{B(\theta )}|h|\, dm +\int _{B(\theta ')}|h|\, dm.
\]
In the latter case, one has
\[
\frac{1}{|A|} \left| \int _Ah \, dm\right| \leq \frac{2}{|A|} \int _{B(\theta )}|h|\, dm\leq 2\frac{|B(\theta )|}{|A|} n\left( 1+\frac{1}{q_{n-1}}\right) 
\]
by Lemma \ref{average_over_B}. 
Since 
\[
|B(\theta )|=b_{|\theta |}b_{2n+1-|\theta |}=b_{n-1}b_{n+2}\leq b_{n+2}
\]
it follows that 
\begin{align} \label{case4i}
\frac{1}{|A|} \left| \int _Ah \, dm\right| \leq 4n\frac{b_{n+2}}{|A|} \leq 4n\frac{b_{n+2}}{b_{n+1}^2} .
\end{align}

\item[4-ii) The case $y\leq b_{n+1}$.] 
One has $xy\ll b_nb_{n+1}=|B_n|$ by \eqref{regularity}. 
Let 
\[
D_4=\bigcup _{A_{xy}\in \mathcal{A} _\mathcal{C}} \left\{ \text{int} A_{xy}\colon (x,y)\in (b_{n+1},b_{n}]\times (0,b_{n+1}] \right\} .
\]
\end{description}
\end{description}

Now we set 
\begin{equation} \label{def_D}
D=D_1\cup D_2\cup D_3\cup D_4. 
\end{equation}
Observe here that $D_4$ is covered by $D_2$, hence one has $D=D_1\cup D_2\cup D_3$. 

\begin{lemma} \label{comparison_D_F}
Given $\varepsilon \in (0,1)$, one can define $F$ and $D$ such that $|D|\leq \varepsilon |F|$. 
\end{lemma}
\begin{proof}
We see that $D_1$ is covered by a rectangle $R$ with height $2(\overline{a} _{n+1}+\tau q_{n-1})$ and width $1$. 
Then by the assumption \eqref{regularity},
\[
|D_1|\leq |R|=2(\overline{a} _{n+1}+\tau q_{n-1})\ll 2(b_nb_{n+1}+\tau q_{n-1})=2(|B_n|+\tau q_{n-1}). 
\]

Next, we estimate $|D_2|$. 
As we have seen in Case 2-ii, the area of each rectangle $A_{xy}$ with $(x,y)\in (b_{r+1},b_{r}]\times (b_{2n-r+1},b_{2n-r}]$ is estimated as 
\[
|A_{xy}|=xy\leq \overline{a} _{r}\overline{a} _{2n-r}\leq \lambda |B_{r+1}|
\]
for each $r\in \{ 1,\dots ,n-1\}$. 
Here, it follows from Lemma \ref{regular_distortion} with $k=n-r$ that 
\[
\overline{a} _{r}\overline{a} _{2n-r} \leq \lambda b_{r+1}b_{2n-r} 
\]
and 
\[
b_{r+1}b_{2n-r} \leq \lambda \overline{a} _{r+1}\underline{a} _{2n-r}. 
\]
Since $\underline{a} _{2n-r} \leq \overline{a} _{2n-r-1}$ by definition, one has
\begin{equation} \label{pairwise_decay_C} 
\overline{a} _{r}\overline{a} _{2n-r}\leq \lambda ^2 \overline{a} _{r+1}\overline{a} _{2n-r-1}
\end{equation}
for each $r\in \{ 1,\dots ,n-1\}$. 
Note also that 
\begin{equation} \label{decay_C} 
\frac{\overline{a} _{r+1}}{\overline{a} _{r}} =\frac{\overline{a} _{r+1}}{b_{r+1}} \frac{b_{r+1}}{\overline{a} _{r}} \leq \frac{\overline{a} _{r+1}}{b_{r+1}} \frac{b_{r+1}}{\underline{a} _{r+1}} \ll 1 
\end{equation}
by \eqref{regularity}. 
By making use of \eqref{decay_C} and \eqref{pairwise_decay_C}, we can apply the same argument for the proof of Lemma \ref{area} to the sequence $\{ \overline{a} _{1},\dots ,\overline{a} _{n-1},\overline{a} _{n+1},\dots ,\overline{a} _{2n-1}\} $. 
As a result, we have  
\begin{align*}
|D_2|
&\leq (n-1)\left| \cup _{A_{xy}\in \mathcal{A} _\mathcal{C}} \left\{ A_{xy}\colon (x,y)\in (b_{r+1},b_{r}]\times (b_{2n-r+1},b_{2n-r}] \right\} \right| \\
&\ll (n-1)\left| \cup _{|\theta |=r+1} B(\theta )\right| \leq (n-1)|B_n|. 
\end{align*}

In the same way as above, we have $|D_3|\ll |B_n| $. 
 
Consequently, we have 
\[
|D|\leq |D_1|+|D_2|+|D_3|\ll |B_n|+(n-1)|B_n|+|B_n|\leq \frac{n+1}{n} K_E|E|
\]
by Lemma \ref{area}-\eqref{areaE}, and thus
\[
|D|\ll K_EK_F\frac{n+1}{n-1} |F|
\]
by \eqref{comparison_E_F}.
\end{proof}
\subsubsection{Proof of Lemma \ref{main_lemma}} \label{proof_main_lemma}

\begin{lemma} \label{rectangle_of_large_area}
Let $R$ be a rectangle with $|R|>4n|B_n|$. 
Then one has
\[
\frac{1}{|R|} \left| \int_R h\, dm\right| <\frac{1}{2}. 
\]
\end{lemma}
\begin{proof}
We have  
\[
\frac{1}{|R|} \left| \int_R h\, dm\right| \leq \frac{1}{|R|} \| h\| _{L^1(m)}=\frac{2}{|R|} n|B_n|<\frac{1}{2}. 
\] 
\end{proof}

\begin{proof}[Proof of Lemma \ref{main_lemma}]
One can suppose, with the aid of \eqref{regularity}, that $b_1>\dots >b_n>b_{n+1}>\dots >b_{2n}$ will satisfy 
\begin{enumerate}
\item Condition \eqref{area_decrease},
\item $\dfrac{b_n}{\underline{a} _n}<\dots <\dfrac{b_1}{\underline{a} _1}<\dfrac{1}{3n2^{n-1}}$, 
\item Lemma \ref{regular_distortion},
\item $\delta >4nb_nb_{n+1}$, 
\item $b_{n+2}<\dfrac{b_{n+1}^2}{4n} $.
\end{enumerate}
For such a sequence, let $F$ and $D$ be the sets and $h$ be the function defined as \eqref{def_F}, \eqref{def_D} and \eqref{def_h}, respectively. 
Then we have Lemma \ref{comparison_D_F}, and the property \eqref{order_ratio_L1norm_areaF} follows from \eqref{ratio_L1norm_areaF}.
Let $\eta =|B_1|$. 
Then $\eta \in (0,\delta )$, and the property (ii) follows from Lemma \ref{average_over_B}. 

The property (i) follows from the consequences of Cases 2-i), 3-i), and 4-i). 
Indeed, we have for \eqref{case2ia} that
\begin{align*}
\frac{1}{|A|} \int _Ah\, dm \leq n\frac{p+2}{p} \frac{b_{r+1}}{\underline{a} _{r+1}} \prod _{j=r+1}^{n-1}\left( 1+\frac{1}{q_j}\right) \leq 3n\frac{1}{3n2^{n-1}} 2^{n-r-1} <1
\end{align*}
for every $r\in \{ 1,\dots ,n-1\}$ by the condition 2. 
Similarly, we have for \eqref{case2ib} that
\begin{align*}
\frac{1}{|A|} \int _Ah\, dm  
&\leq 2n\frac{b_{r+1}}{\underline{a} _{r+1}} \prod _{j=r+1}^{n-1}\left( 1+\frac{1}{q_j}\right) <1
\end{align*}
for every $r\in \{ 1,\dots ,n-1\}$ by the condition 2. 
For \eqref{case4i}, it follows from by the condition 5 that 
\begin{align*}
\frac{1}{|A|} \int _Ah\, dm \leq 4n\frac{b_{n+2}}{b_{n+1}^2} <1. 
\end{align*}

Lemma \ref{rectangle_of_large_area} yields the property (iii). 
Lemma \ref{main_lemma} is obtained.
\end{proof}
\section{Proof of Theorem \ref{main-3}}

\subsection{Random translations}

Let 
$$
\omega =(\alpha, \beta),
$$
where $\alpha$ and $\beta$ are uniformly distributed on $[0,1]^2$.
Let $\omega_1, \dots, \omega_N$ be a set of independent, uniformly distributed vectors on $[0,1]^2$. 
For $z\in [0,1]^2$ and $\omega =(\omega _1, \dots, \omega _N)$, consider the product
\[
f_N(z,\omega )=\prod _{k=1}^N \left( 1-\mathbbm{1} _F(z+\omega _k)\right) .
\]
Note that $f_N(z,\omega )=0$ if and only if $1-\mathbbm{1} _F(z+\omega _k)=0$ for some $k$, i.e., $z\in F+ \omega _k$. 
Hence $\supp f_N(\cdot ,\omega )\subset [0,1]^2$ is the set which is not covered by the set 
\[
F_0(\omega )=\bigcup _{k=1}^N\{ F+\omega_k\}.
\] 
Therefore, the support of $1-f_N(\cdot ,\omega )$ will be the set that is covered by the set $F_0(\omega )$. 
Next, consider also
\[
g_N(z,\omega )=\prod_{k=1}^N \left( 1-\mathbbm{1} _D(z+ \omega _k)\right) ,
\]
and denote
\[
D_0(\omega )=\bigcup _{k=1}^N\{ D+\omega _k\}.
\] 
Note that the support of the function $\left( (1-f_N)\cdot g_N\right) (\cdot ,\omega )$ is the set of points that is covered by $F_0(\omega )$ but not by $D_0(\omega )$, namely $F_0(\omega )\setminus D_0(\omega )$. 
Denote $Q_0(\omega )=F_0(\omega )\setminus D_0(\omega )$. 
For each $\omega =(\omega _1, \dots, \omega _N)$, let $|Q_0|(\omega )=|Q_0(\omega )|$. 
Then we have
\[
|Q_0|(\omega )=\int _{[0,1]^2}\left( (1-f_N)\cdot g_N\right) (z,\omega )\, dm(z).
\]
Below, we denote the integration with respect to $\omega =(\omega _1, \dots, \omega _N)$ by $\mathcal{E}$ for notational simplicity. 
Namely, 
\begin{equation} \label{integration_E}
\mathcal{E} (u)=\int _{([0,1]^2)^N}u(\omega )\, d\omega =\int_{[0,1]^2} \dots \int _{[0,1]^2}u(\omega _1,\dots ,\omega _N)\, d\omega _1 \dots d\omega _N
\end{equation}
for a function $u$ on $([0,1]^2)^N$. 

\begin{remark} 
Precisely, the integral \eqref{integration_E} above should be written as
\[
\mathcal{E}(u)=\int _Yu(\omega (y))\, dP(y)=\int _{([0,1]^2)^N}u\, dP^\omega 
\]
where $(Y,P)$ denotes the underlying probability space on which the $N$-tuple $\omega =(\omega _1, \dots, \omega _N)$ of independent random variables is defined, and $P^\omega $ is the joint distribution. 
By independency,  
\begin{equation} \label{integration_P}
\mathcal{E} (u)=\int _{([0,1]^2)^N}u\, dP^\omega =\int_{[0,1]^2} \dots \int _{[0,1]^2}u\, dP^{\omega _1} \dots dP^{\omega _N},
\end{equation}
and denote \eqref{integration_P} by \eqref{integration_E} for an abuse of notation. 
\end{remark}

Henceforth, we construct a function defined on (a neighborhood of) a set $Q_0(\omega )=F_0(\omega )\setminus D_0(\omega )$ with suitably chosen $\omega $. 
Here we have the following lemma on such $\omega$. 

\begin{lemma} \label{Omega_chi}
Let $N=\lceil1/|F|\rceil \in \mathbb{N}$. 
There is an open set $\Omega \subset ([0,1]^2)^{\lceil1/|F|\rceil }$ with $|\Omega |>0$ such that for any $\omega \in \Omega $ we have
\[
|Q_0|(\omega )=|F_0\setminus D_0|(\omega )\geq \frac{99}{100} \left( \frac{1}{(2e)^{\varepsilon } }-\frac{1}{e} \right) >0
\]
for sufficiently small $\varepsilon \in (0,1)$. 
\end{lemma}

In fact, we will prove a more general result which will be of independent interest. 
See also Remark \ref{remark_Omega_chi_general_form}.

\begin{proposition} \label{Omega_chi_general_form}
Let $A,B\subset [0,1]^2$ be sets such that $0<|B|\leq c|A|$ for some $c>0$, and let $N=\lceil1/|A|\rceil \in \mathbb{N}$. 
Suppose that for any $\omega \in ([0,1]^2)^N$ we have $|A\cap (B+\omega)|\leq \varepsilon |A|$ for some $\varepsilon \in \left( 0,\min \{ 1,c_0\} \right)$ where $c_0=|B|/|A|\in (0,c]$. 
Then there exist $\kappa =\kappa (c_0|A|)\in (0,1)$ and an open set $\Omega \subset ([0,1]^2)^N$ with $|\Omega |>0$ such that for any $\omega \in \Omega $ we have
\[
\left| \bigcup _{k=1}^{N}\{A+\omega _k\} \setminus \bigcup _{k=1}^{N}\{B+\omega _k\}\right| >\frac{99}{100} \left( \frac{1}{(\kappa ^{-1}e)^{c_0}}-\frac{1}{e^{(1+c_0-\varepsilon )}} \right) .
\]
\end{proposition}

Note that $\kappa =\kappa (c_0|A|)\in (0,1)$ in Proposition \ref{Omega_chi_general_form} can be taken such that $\kappa \to 1$ as $c_0|A|\to 0$. 
Once Proposition \ref{Omega_chi_general_form} is obtained, we have Lemma \ref{Omega_chi}. 
\begin{proof}[Proof of Lemma \ref{Omega_chi}] 
Take $A=F$ and $B=D$ in Proposition \ref{Omega_chi_general_form}. 
One can take $c=\varepsilon $ by Lemma \ref{comparison_D_F}.  
Hence Proposition \ref{Omega_chi_general_form} yields the result with $\kappa =1/2$. 
\end{proof}

We postpone the proof of Proposition \ref{Omega_chi_general_form} for a short while, and proceed the argument.  
Henceforth, we let $N=\lceil1/|F|\rceil $. 
Fix sufficiently small $\varepsilon \in (0,1)$, and set 
\[
\chi =\frac{99}{100} \left( \frac{1}{(2e)^{\varepsilon } }-\frac{1}{e} \right) >0.
\]
As a result, we obtain the following lemma. 

\begin{lemma} \label{random_translation}
Let $\delta >0$ and $n\in \mathbb{N}$. 
There exist a function $h_0\in L^1([0,1]^2,m)$ with
\[
\| h_0\| _{L^1}\leq 4K_EK_F\frac{n}{n-1} ,
\]
a set $Q_0\subset [0,1]^2$ with $|Q_0|>\chi $, and $\eta =\eta (\varepsilon ,\delta ,n)\in (0,\delta )$ such that we have the following.
\begin{enumerate}
\item[{\rm (i)}] If $z \in Q_0$ and  $z\in A\in \mathcal{R}_\mathcal{C}$, then
\[
\left| \frac{1}{|A|}\int _A h_0 \, dm\right| \leq 2.
\]
\item[{\rm (ii)}] If $z \in Q_0$, then there exists $B\in \mathcal{R}_\mathcal{D}$ with $B\ni z$ and $|B|>\eta $ such that
\[
\left| \frac{1}{|B|}\int _B h_0\, dm\right| >n. 
\]
\item[{\rm (iii)}] For every $R \in \mathcal{R}$ with $|R|>\delta $, we have
\[
\left| \frac{1}{|R|}\int_R h\, dm\right| <1.
\]
\end{enumerate}
\end{lemma}
\begin{proof}
Let $\Omega $ be the set as in Lemma \ref{Omega_chi}. 
Define 
\[
\Omega _0=\left\{ \omega \in \Omega \colon |\omega _i-\omega _j|\notin \{ b_1,\dots ,b_n\} \right\} .
\]
We see $|\Omega _0|=|\Omega |>0$. 

Since $\omega \mapsto Q_0(\omega )=F_0(\omega )\setminus D_0(\omega )$ is continuous with respect to the Hausdorff metric on the space of all closed subsets of $[0,1]^2$, there exists $\omega _0=((\omega _0)_1,\dots ,(\omega _0)_N)\in \Omega _0$ such that for any rectangle $A=A_{xy}$ with sides $x,y \in \mathcal{C}$ containing $z\in Q_0(\omega _0)$, there will be cancellations for all but at most two rectangles from $\mathcal{R} _D$, say $B=B(\theta )+(\omega _0)_i$ and $B'=B(\theta ')+(\omega _0)_j$ for some $\theta ,\theta '\in \Theta $ and $i,j\in \{ 1,\dots ,N\}$. 
Namely, let 
\[
h_0=h_{F+(\omega _0)_1}+\cdots +h_{F+(\omega _0)_N},
\]
then it follows that either $\int _Ah_0 \, dm=0$ or
\begin{align*}
\left| \frac{1}{|A|}\int _A h_0 \, dm\right| 
&\leq \left| \frac{1}{|A|}\int _{A\cap B} h_{F+(\omega _0)_i} \, dm(z)\right| +\left| \frac{1}{|A|}\int _{A\cap B'} h_{F+(\omega _0)_j} \, dm\right| \leq 2
\end{align*}
by Lemma \ref{main_lemma}-(i). 
Hence the first property follows. 

The argument same as above shows the third property by Lemma \ref{main_lemma}-(iii). 

Since $\omega _0\in \Omega _0$, Lemma \ref{main_lemma}-(ii) yields the property (ii) by taking $\tau >0$ small if necessary. 
(Recall that $\tau >0$ determines the height of the support of $h_0$.) 

One has
\begin{align*}
\| h_0\| _{L^1(m)}
&\leq \left\| h_{F+(\omega _0)_1}\right\| _{L^1(m)} +\cdots +\left\| h_{F+(\omega _0)_N}\right\| _{L^1(m)} \\
&\leq 2K_EK_F|F|\frac{n}{n-1} \left( \frac{1}{|F|} +1\right) \leq 4K_EK_F\frac{n}{n-1}
\end{align*}
by \eqref{ratio_L1norm_areaF} and \eqref{comparison_E_F} with noting $N=\lceil1/|F|\rceil $. 
Hence Lemma holds for $Q_0=Q_0(\omega _0)$ and $h_0$ defined as above.  
\end{proof}
\subsubsection{Proof of Proposition \ref{Omega_chi_general_form}}

In this section, we denote $X=[0,1]^2$ for notational simplicity. 
Given $\omega =(\omega _1, \dots, \omega _N)\in X^N$, denote
\[
A_0(\omega )=\bigcup _{k=1}^{N}\{ A+\omega _k\} \quad \text{and} \quad B_0(\omega )=\bigcup _{k=1}^{N}\{ B+\omega _k\} .
\]
For each $\omega =(\omega _1, \dots, \omega _N)$, let $|A_0\setminus B_0|(\omega )=|A_0(\omega )\setminus B_0(\omega )|$. 

\begin{lemma} \label{lem:lower_estimate_N}
We have 
\begin{equation} \label{lower_estimate_N}
\mathcal{E} \left( |A_0\setminus B_0|\right) \geq \left( 1-c_0|A|\right) ^{N}-\left( 1-(1+c_0-\varepsilon )|A|\right) ^{N} .
\end{equation}
\end{lemma}
\begin{proof}
We see
\begin{align} \label{integral0}
\mathcal{E} \left( |A_0\setminus B_0|\right)  
&=\int_{X^N}|A_0\setminus B_0|(\omega )\,  d\omega  \nonumber \\
&=\mathcal{E} \left( \int _{X}b_N(z,\omega )\, dm(z) \right) -\mathcal{E} \left( \int _{X}(a_N\cdot b_N)(z,\omega )\, dm(z)\right) ,
\end{align}
where 
\[
a_N(z,\omega )=\prod _{k=1}^N \left( 1-\mathbbm{1} _A(z+\omega _k)\right) \quad \text{and} \quad b_N(z,\omega )=\prod _{k=1}^N \left( 1-\mathbbm{1} _B(z+\omega _k)\right) .
\]
For the first term in the right-hand side of \eqref{integral0}, by the Fubini theorem, 
\begin{align*}
\mathcal{E} \left( \int _{X}b_N(z,\omega )\, dm(z) \right) 
&=\mathcal{E} \left( \int _{X} \prod _{k=1}^N(1-\mathbbm{1} _B(z+\omega _k))\, dm (z)\right) \\
&=\int _{X} \mathcal{E} \left( \prod _{k=1}^N(1-\mathbbm{1} _B(z+\omega _k))\right) \, dm (z).
\end{align*}
Since $|B|=c_0|A|$ by assumption, we obtain
\[
\mathcal{E} \left( \prod _{k=1}^N\left( 1-\mathbbm{1} _B(z+\omega _k)\right) \right) =(1-|B|)^{N}=(1-c_0|A|)^{N}
\]
for every $z\in X$, and thus 
\begin{equation} \label{integral_g}
\mathcal{E} \left( \int _{X}b_N(z,\omega )\, dm(z) \right) =(1-c_0|A|)^{N}.
\end{equation}

Next, for the second term in the right-hand side of \eqref{integral0}, by the Fubini theorem again, one has
\[
\mathcal{E} \left( \int_{X}(a_N\cdot b_N)(z,\omega )\, dm(z) \right) =\int _{X}\mathcal{E} (a_N\cdot b_N)(z)\, dm(z) ,
\]
where 
\begin{align}
\mathcal{E} (a_N\cdot b_N)(z)
&=\mathcal{E} \left( \left( \prod_{k=1}^N\left( 1-\mathbbm{1} _A(z+ \omega _k)\right) \right) \prod _{k=1}^N\left( 1-\mathbbm{1} _B(z+\omega _k)\right) \right) \nonumber \\
&=\mathcal{E} \left( \prod _{k=1}^N\left( 1-\mathbbm{1} _A(z+\omega _k)-\mathbbm{1} _B(z+\omega _k)+\mathbbm{1} _{A\cap B}(z+\omega _k)\right) \right) \nonumber \\
&=\left( 1-|A|-|B|+|A\cap B|\right) ^N \nonumber \\
&\leq \left( 1-|A|-c_0|A|+\varepsilon |A|\right) ^N \label{c_epsilon}
\end{align}
for every $z\in [0,1]^2$. 
Hence we have
\begin{equation} \label{integral_fg}
\mathcal{E} \left( \int _{X}(a_N\cdot b_N)(z,\omega )\, dm(z)\right) \leq \left( 1-(1+c_0-\varepsilon )|A|\right) ^N 
\end{equation}
Lemma \ref{lem:lower_estimate_N} follows from \eqref{integral0}, \eqref{integral_g} and \eqref{integral_fg}.
\end{proof}

\begin{proof}[Proof of Proposition \ref{Omega_chi_general_form}]
Since $1/|A|\leq N<(1+|A|)/|A|$, it follows from Lemma \ref{lem:lower_estimate_N} that 
\[
\left( 1-c_0|A|\right) ^{N}\geq \left( 1-c_0|A|\right) ^{(1+|A|)/|A|}\geq \left( \kappa \cdot e^{-1}\right) ^{c_0}
\]
for some $\kappa =\kappa (c_0|A|)\in (0,1)$ such that $\kappa \to 1$ as $c_0|A|\to 0$, and 
\[
\left( 1-(1+c_0-\varepsilon )|A|\right) ^{N}\leq \left( 1-(1+c_0-\varepsilon )|A|\right) ^{1/|A|}\leq e^{-(1+c_0-\varepsilon )}.
\]
Hence, by Lemma \ref{lem:lower_estimate_N}, we arrive at 
\begin{equation} \label{lower_estimate_epsilon}
\mathcal{E} \left( |A_0\setminus B_0|\right) =\int_{X^N}|A_0\setminus B_0|(\omega )\,  d\omega \geq \frac{1}{(\kappa ^{-1}e)^{c_0} }-\frac{1}{e^{(1+c_0-\varepsilon )}} .
\end{equation}

Since $\omega \mapsto  |A_0\setminus B_0|(\omega )$ is a non-negative and continuous function, it follows from \eqref{lower_estimate_epsilon} that there is an open set $\Omega \subset X^N=X^{\lceil1/|A|\rceil }$ with $|\Omega |>0$ such that for any $\omega \in \Omega $ we have
\[
|A_0\setminus B_0|(\omega )\geq \frac{99}{100} \left( \frac{1}{(\kappa ^{-1}e)^{c_0}}-\frac{1}{e^{(1+c_0-\varepsilon )}} \right) .
\]
Proposition is obtained. 
\end{proof}

\begin{remark} \label{remark_Omega_chi_general_form}
{\rm Proposition \ref{Omega_chi_general_form} concerns the case where $c_0>0$ (or $c$) is rather large. 
If $c>0$ is small, this is exactly the case of Lemma \ref{Omega_chi}, then we have a better upper bound 
\[
\mathcal{E} (a_N\cdot b_N)(z)\leq \left( 1-|A|\right) ^N
\]
in stead of \eqref{c_epsilon}, and thus 
\[
\mathcal{E} \left( |A_0\setminus B_0|\right) \geq \left( 1-c|A|\right) ^{N}-\left( 1-|A|\right) ^{N} 
\]
in stead of \eqref{lower_estimate_N}. 
}
\end{remark}
\subsection{Proof of Theorem \ref{main-3}}

\begin{proof}
Denote $X_0=[0,1]^2$, and let $h_0\in L^1(X_0)$ and $Q_0=F_0(\omega _0)\setminus D_0(\omega _0)\subset X_0$ be as in Lemma \ref{random_translation}. 
Hence $|Q_0|>\chi $ by Lemma \ref{random_translation}. 
Let $h_0^*=|Q_0|h_0$. 
Then $\| h_0^*\| _{L^1(m)}\leq (4+o(1))|Q_0|$. 
Since we fix such an $\omega _0\in \Omega _0$, we will omit $\omega _0$ and write, say $D_0$ instead of $D_0(\omega _0)$ for notational simplicity. 
By Lemma \ref{comparison_D_F}, one has
\[
|D_0|\leq |D|\left( \frac{1}{|F|} +1\right) <\varepsilon _0=\frac{\varepsilon }{2^3} .
\]
Here and below, we let $\varepsilon _{k}=\varepsilon /2^{k+3} $ for $k\in \mathbb{N} \cup \{ 0\} $. 

Let 
\[
Y_1=X_0\setminus \left( Q_0\sqcup D_0\right) =X_0\setminus \left( F_0\cup D_0\right) .
\]
Since the boundary of the set $Q_0\sqcup D_0$ consists of straight lines, we can partition $Y_1$ into squares of small size. 
Precisely, there exist $a_1=a_1(\varepsilon )\in (0,1)$ and a finite family of squares $S_{1,a}$ with side lengths $a\in [a_1,1)$ such that 
\[
\left| Y_1\setminus \bigsqcup _{a\geq a_1}S_{1,a}\right| <\varepsilon _{0}.
\]
(Note that the disjoint union includes squares of same side lengths.)
Let 
\[
X_1=\bigsqcup _{a\geq a_1}S_{1,a}.
\]
Inside each square $S_{1,a}$, we can find a function $h_{1,a}\in L^1(X_0)$ and a subset $Q_{1,a}=F_{1,a}\setminus D_{1,a}\subset S_{1,a}$ for which all the properties of Lemma \ref{random_translation} hold, and hence 
\[
\left| \bigsqcup _{a\geq a_1}Q_{1,a}\right| >\chi \left| \bigsqcup _{a\geq a_1}S_{1,a}\right| =\chi |X_1|. 
\]
Define 
\[
Q_1^*=\bigsqcup _{a\geq a_1}Q_{1,a}\quad \text{ and } \quad D_1^*=\bigsqcup _{a\geq a_1}D_{1,a}. 
\]
Then $Q_1^*,D_1^*\subset X_1$, and one has 
\[
|D_1^*|<\varepsilon _{1}.
\]
Notice that each support $\supp h_{1,a}$ consists of finitely many horizontal line segments. 
Hence we may and do assume that the supports $\supp h_{1,a}$ in different squares $S_{1,a}$ never lay on the same horizontal line. 
We may do assume further that no two such line segments, one from $\supp h_0^*$ and the other from any of $\supp f_{1,a}$, lay on the same horizontal line. 
In short, one can assume that any two horizontal line segments in $\supp h_0^*\sqcup \{ \supp h_{1,a}\colon a\geq a_1\} $ have a ``vertical gap''. 
Let 
\[
h_1^*=\sum _{a\geq a_1}|Q_{1,a}|h_{1,a}\in L^1(X_0). 
\]
Then, by construction, we still have all the properties of Lemma \ref{random_translation} for $Q_0\sqcup Q_1^*$ and $h_0^*+h_1^*$ by the same argument of the proof for Lemma \ref{random_translation}. 
Here each function $h_{1,a}$ is multiplied by the area $|Q_{1,a}|$ just to have $\| h_1^*\| _{L^1(m)}\leq (4+o(1))|Q_1^*|$. 
Hence it follows that $\| h_0^*+h_1^*\| _{L^1(m)}\leq (4+o(1))(|Q_0|+|Q_1^*|)$. 

Next, define 
\[
Y_2=X_1\setminus \left( Q_1^*\sqcup D_1^*\right) ,
\]
and repeat the procedure described above. 
Namely, there exist $a_2=a_2(\varepsilon )\in (0,a_1)$ and a finite family of squares $S_{2,a}$ with side lengths $a\in [a_2,a_1)$ with 
\[
\left| Y_2\setminus \bigsqcup _{a\in [a_2,a_1)}S_{2,a}\right| <\varepsilon _{1}
\]
such that for each $S_{2,a}$, there are $h_{2,a}\in L^1(X_0)$ and $Q_{2,a}=F_{2,a}\setminus D_{2,a}\subset S_{2,a}$ for which all the properties of Lemma \ref{random_translation} hold such that 
\[
\left| \bigsqcup _{a\in [a_2,a_1)}Q_{2,a}\right| >\chi \left| \bigsqcup _{a\in [a_2,a_1)}S_{2,a}\right| \quad \text{ and } \quad \left| \bigsqcup _{a\in [a_2,a_1)}D_{2,a}\right| <\varepsilon _{2}.
\]
Define 
\begin{align*}
X_2=\bigsqcup _{a\in [a_2,a_1)}S_{2,a}, \quad Q_2^*=\bigsqcup _{a\in [a_2,a_1)}Q_{2,a}, \quad D_2^*=\bigsqcup _{a\in [a_2,a_1)}D_{2,a}. 
\end{align*}
Note that one has $Q_2^*,D_2^*\subset X_2$, and $|D_2^*|<\varepsilon _{2}$. 
We let 
\[
h_2^*=\sum _{a\in [a_2,a_1)}|Q_{2,a}|h_{2,a}\in L^1(X_0),
\]
where $\supp h_{2,a}$ are placed such that any two horizontal line segments in $\supp h_0^*\sqcup \supp h_1^*\sqcup \supp h_2^*$ have a vertical gap. 
Hence, all the properties of Lemma \ref{random_translation} still hold for $Q_0\sqcup Q_1^*\sqcup Q_2^*$ and $h_0^*+h_1^*+h_2^*$. 
Notice that $\| h_0^*+h_1^*+h_2^*\| _{L^1(m)}\leq (4+o(1))(|Q_0|+|Q_1^*|+|Q_2^*|)$ holds. 

Once $X_k$, and $Q_k^*,D_k^*\subset X_k$ are defined, by letting 
\[
Y_{k+1}=X_k\setminus \left( Q_k^*\sqcup D_k^*\right) ,
\]
one can define $X_{k+1}\subset Y_{k+1}$, and $Q_{k+1}^*,D_{k+1}^*\subset X_{k+1}$ with 
\[
\left| Y_{k+1}\setminus X_{k+1}\right| <\varepsilon _{k},\quad \left| Q_{k+1}^*\right| >\chi \left| X_{k+1}\right| ,\quad \quad |D_{k+1}^*|<\varepsilon _{k+1} 
\]
and a function $h_{k+1}^*\in L^1(X_0)$, with $\| h_{k+1}^*\| _{L^1(m)}\leq (4+o(1))|Q_{k+1}^*|$, having the vertical gap property among $\supp h_0^*\sqcup \supp h_1^*\sqcup \dots \sqcup \supp h_{k+1}^*$ such that all the properties of Lemma \ref{random_translation} hold for $Q_0\sqcup Q_1^*\sqcup \dots \sqcup Q_{k+1}^*$ and $h_0^*+h_1^*+\dots +h_{k+1}^*$. 

Now, we show that the area of $\bigsqcup _{k=0}^{n-1}Q_k^*$, where $Q_0^*=Q_0$, can be arbitrarily close to one. 
\begin{lemma} \label{area_E}
One can take $n\in \mathbb{N}$ so large that 
\[
\left| \bigsqcup _{k=0}^{n-1}Q_k^*\right| >1-\varepsilon .
\]
\end{lemma}
\begin{proof}
One has
\[
|Q_0^*|=|Q_0|=|X_0|-|D_0|-|Y_1|>1-\varepsilon _0-(|X_1|+\varepsilon _0)=1-|X_1|-2\varepsilon _0,
\]
and 
\begin{align*}
|Q_0^*\sqcup Q_1^*|=|Q_0^*|+|Q_1^*|
&>(1-|X_1|-2\varepsilon _0)+(|X_1|-|D_1^*|-|Y_2|) \\
&>1-2\varepsilon _0-\varepsilon _1-|Y_2| \\
&>1-2\varepsilon _0-\varepsilon _1-(|X_2|+\varepsilon _1) \\
&>1-|X_2|-2\varepsilon _0-2\varepsilon _1.
\end{align*}
Hence by induction, one obtain
\[
\left| \bigsqcup _{k=0}^{n-1}Q_k^*\right| >1-|X_{n}|-2\sum _{k=0}^{n-1}\varepsilon _k=1-|X_{n}|-2\sum _{k=0}^{n-1}\frac{\varepsilon }{2^{k+3}}>1-|X_{n}|-\frac{\varepsilon }{2}.
\]

Next, one sees that $|X_{n}|$ strictly decreases to 0 as $n$ grows. 
Indeed, 
\begin{align*}
|X_{n+1}|\leq |Y_{n+1}|
&=|X_n|-|Q_n^*|-|D_n^*| \\
&<|X_n|-|Q_n^*|\\ 
&<|X_n|-\chi |X_n| \\
&=(1-\chi )|X_n|<\dots<(1-\chi )^{n+1}|X_0|=(1-\chi )^{n+1}. 
\end{align*}

Take $n\in \mathbb{N}$ so large that $(1-\chi )^{n}<\varepsilon /2$. 
Then it follows that 
\[
\left| \bigsqcup _{k=0}^{n-1}Q_k^*\right| >1-|X_{n}|-\frac{\varepsilon }{2} >1-\varepsilon .
\]
Lemma is obtained. 
\end{proof}

Take an $n\in \mathbb{N}$ so large as in Lemma \ref{area_E}. 
Letting $Q=\bigsqcup _{k=0}^{n-1}Q_k^*$ and $f=h_0^*+h_1^*+\dots +h_{n-1}^*$ yields the Theorem \ref{main-3}. 
\end{proof}
\section{Proof of Theorem \ref{main-1}}

\subsection{Setup}

To prove Theorem \ref{main-1}, we use Theorem \ref{main-3} recursively to find sequences of sets $\{ Q_n\}_{n\geq 1}$ and functions $\{f_n\}_{n\geq 1}$ as follows. 
Let $\varepsilon _1=1/2$, $\delta _1=1/2$, and $a_1=1$. 
Then there are $Q_1\subset [0,1]^2$ with $|Q_1|>1-\varepsilon _1$ and $f_1\in L^\infty (m)$ with $\| f_1\| _{L^1(m)}\leq 2$ and $\eta _1\in (0,\delta _1)$ such that 
\begin{enumerate}
\item[(i$_1$)] If $z \in Q_1$ and $z \in A \in \mathcal{R}_\mathcal{C}$, then
\[
\left| \frac{1}{|A|}\int_A f_1\, dm\right| \leq 2.
\]
\item[(ii$_1$)] If $z \in Q_1$, then there exists $B \in \mathcal{R}_\mathcal{D}$, with $z \in B$ and $|B|>\eta _1$, so that
\[
\left| \frac{1}{|B|} \int_B f_1\, dm\right| >a_1(=1).
\]
\item[(iii$_1$)]
For every $R \in \mathcal{R}$ with $|R|>\delta _1$, we have
\[
\left| \frac{1}{|R|}\int_R f_1\, dm\right| <1.
\]
\end{enumerate}

Suppose that $Q_{n-1}\subset [0,1]^2$, $f_{n-1}\in L^\infty (m)$ and $\eta _{n-1}>0$ are defined such that the corresponding properties (i$_{n-1}$), (ii$_{n-1}$), (iii$_{n-1}$) hold. 
Let $\varepsilon _n=1/(n+1)^2$, $\delta _n=\eta _{n-1}$, and 
\[
a_n=\left( 2\sup _{k\leq n-1}\| f_k\| _{L^\infty (m)}+n\right) n^2. 
\]
Then there are $Q_n\subset [0,1]^2$ with $|Q_n|>1-\varepsilon _n$ and $f_n\in L^\infty (m)$ with $\| f_n\| _{L^1(m)}\leq 2$ and $\eta _n\in (0,\delta _n)$ such that 
\begin{enumerate}
\item[(i$_{n}$)] If $z \in Q_n$ and $z \in A \in \mathcal{R}_\mathcal{C}$, then
\[
\left| \frac{1}{|A|}\int_A f_n\, dm\right| \leq 2.
\]
\item[(ii$_{n}$)] If $z \in Q_n$, then there exists $B \in \mathcal{R}_\mathcal{D}$, with $z \in B$ and $|B|>\eta _n$, so that
\[
\left| \frac{1}{|B|} \int_B f_n\, dm\right| >a_n.
\]
\item[(iii$_{n}$)]
For every $R \in \mathcal{R}$ with $|R|>\delta _n$, we have
\[
\left| \frac{1}{|R|}\int_R f_n\, dm\right| <1.
\]
\end{enumerate}
Notice that $\delta _n>\eta _n=\delta _{n+1}>\eta _{n+1}=\delta _{n+2}$ by construction, and thus both $\eta _n$ and $\delta _n$ are decreasing sequences in particular. 

\subsection{Proof of Theorem \ref{main-1}}

\begin{proof}[Proof of Theorem \ref{main-1}]
Define
\[
f=\sum_{n=1}^\infty \frac{1}{n^2} f_n.
\]
Note that by property \eqref{c-1} of Theorem \ref{main-3}, we have
\[
\left\| \sum_{n=1}^\infty \frac{1}{n^2}f_n\right\| _{L^1(m)}\leq \sum_{n=1}^\infty \frac{1}{n^2}\|f_n\|_{L^1(m)}<\infty.
\]
Hence $f$ is well defined and $f\in L^{1}(m)$.
By the Borel-Cantelli lemma, we also have that
\[
m\left( \liminf _{n\to \infty }Q_{n}\right) =1.
\]
Since $f_k \in L^\infty (m)$, then by Theorem \ref{J-M-Z} we have that for all $k\in \mathbb{N}$, there is $\Gamma _k\subset [0,1]^2$ with $m(\Gamma _k)=1$ such that 
\begin{equation} \label{L1}
\lim_{\substack{\mathrm{diam}R\to 0, \\ z \in R \in \mathcal{R}}}\frac{1}{|R|}\int_R f_k \, dm=f_k(z)
\end{equation}
for every $z\in \Gamma _k$. 
It follows from \eqref{L1} and Theorem \ref{main-3}-(i) that for every $z\in Q_k\cap \Gamma _k$ we have 
\begin{equation} \label{absolute_value}
|f_k(z)|\leq 2.
\end{equation}
Denote $\Gamma _\infty =\bigcap _{k=1}^\infty \Gamma _k$ and 
\[
\Lambda =\left( \liminf _{n\to \infty }Q_{n}\right) \cap \Gamma _\infty .
\]
Clearly $m(\Lambda )=1$.

First, we prove convergence, namely $\delta_{\mathcal{R}_\mathcal{C}}(z,f)=0$ for $m$-almost every $z \in [0,1]^2$. 
Let $z \in \Lambda $. 
Hence there exists $M=M_z\in \mathbb{N}$ so that $z \in Q_n$ for all $n \geq M$. 
Then for every $A\in \mathcal{R}_\mathcal{C}$ with $A\ni z$ we have
\begin{align}
\left| \frac{1}{|A|}\int_A f\, dm-f(z)\right| \nonumber
&=\left| \sum_{n=1}^\infty\frac{1}{n^2}\frac{1}{|A|}\int _A  f_n\, dm-\sum_{n=1}^\infty \frac{f_n(z)}{n^2}\right| \nonumber \\
&\leq
\left| \sum_{n=1}^M \frac{1}{n^2}\frac{1}{|A|}\int _A f_n\, dm - \sum_{n=1}^M \frac{f_n(z)}{n^2}\right| \label{leading} \\
&\quad +\sum_{n=M+1}^\infty \frac{1}{n^2}\left| \frac{1}{|A|} \int _A f_n\, dm\right| +\sum_{n=M+1}^\infty \frac{|f_n(z)|}{n^2}. \label{reminder}
\end{align}
Since $z\in Q_n\cap \Gamma _\infty $ for all $n \geq M$, it follows from \eqref{absolute_value} that $|f_n(z)|\leq 2$ for all $n \geq M$. 
Note also that 
\[
\left| \frac{1}{|A|}\int _A f_n\, dm\right| \leq 2
\]
by the property (i$_{n}$). 
Thus for the last two terms \eqref{reminder}, we have 
\[
\sum_{n=M+1}^\infty \frac{1}{n^2} \left| \frac{1}{|A|} \int _A f_n\, dm\right| +\sum_{n=M+1}^\infty \frac{|f_n(z)|}{n^2}\leq \sum_{n=M}^\infty \frac{4}{n^2}<\frac{4}{M-1}.
\]
For the first term \eqref{leading}, since $x \in \Gamma _\infty $, then we have by \eqref{L1} that
\[
\lim_{\substack{\mathrm{diam}A\to 0, \\ z \in A \in \mathcal{R}_\mathcal{C}}} \left|\sum_{n=1}^M \frac{1}{n^2}\frac{1}{|A|} \int _A f_n\, dm - \sum_{n=1}^M \frac{f_n(z)}{n^2}\right|=0.
\]
Consequently, it follows that
\[
\lim_{\substack{\mathrm{diam}A\to 0, \\ z \in A \in \mathcal{R}_\mathcal{C}}} \left|\sum_{n=1}^\infty \frac{1}{n^2}\frac{1}{|A|}\int_A f_n\, dm - \sum_{n=1}^\infty \frac{f_n(z)}{n^2}\right|=0.
\]
In other words $\delta_{\mathcal{R}_\mathcal{C}}(z,f)=0$ for $m$-almost every $z \in [0,1]^2$.

We now prove divergence. 
Let $z \in \Lambda$. 
Hence there exists $M_z\in \mathbb{N}$ so that $z \in Q_n$ for all $n \geq M$. 
Take $N>M_z$. 
Then by the property (ii$_{N}$), there exists $B \in \mathcal{R}_\mathcal{D}$, with $z \in B$ and $|B|>\eta _N$ such that 
\[
\left|\frac{1}{|B|}\int_B f_N\, dm\right| >a_N.
\]
Hence for such a rectangle $B\in \mathcal{R}_\mathcal{D}$, we have
\begin{align*}
\left| \frac{1}{|B|} \int _B f\, dm\right| 
&=\left| \frac{1}{|B|} \int _B \sum_{n=1}^\infty \frac{1}{n^2} f_n\, dm\right| \\
&> \left|\frac{1}{|B|}\int_B \frac{1}{N^2} f_N\, dm\right| -\left|\sum_{n\neq N}^\infty\frac{1}{n^2}\frac{1}{|B|}\int_B f_n\, dm\right| \\
&>\frac{a_N}{N^2} - \left|\sum_{n\neq N}^\infty\frac{1}{n^2}\frac{1}{|B|}\int_B f_n\, dm\right|.
\end{align*}
Here we have
\begin{align*}
\sum_{n=1}^{N-1} \frac{1}{n^2} \frac{1}{|B|}\int_B |f_n|\, dm\leq \left( \sup_{n\leq N-1}\| f_n\| _{L^\infty (m)}\right) \cdot \sum_{n=1}^{N-1} \frac{1}{n^2} \leq 2 \sup_{n\leq N-1}\|f_n\|_{L^\infty (m)}.
\end{align*}
Notice that we have $|B|>\delta _{n}$ for all $n>N$ since $|B|>\eta _N=\delta_{N+1}$ and $\delta _k$ is a decreasing sequence in $k$. 
Thus by the property (iii$_{n}$) with $n>N$, we have 
\[
\left| \frac{1}{|B|}\int_B f_n\, dm\right| \leq 1
\]
for all $n>N$, and hence
\[
\left|\sum_{n>N}^{\infty} \frac{1}{n^2} \frac{1}{|B|}\int_B f_n\, dm \right|\leq \sum_{n=N+1}^{\infty} \frac{1}{n^2}<\frac{1}{N}.
\]
It follows that 
\[
\left| \frac{1}{|B|}\int _B f\, dm\right| >\frac{a_N}{N^2} - 2 \sup _{n\leq N-1}\| f_n\| _{L^\infty (m)} - \frac{1}{N} =N-\frac{1}{N}.
\]
Consequently, we have
\[
\left| \frac{1}{|B|}\int _B f\, dm-f(z)\right| \geq \left| \frac{1}{|B|}\int _B f\, dm\right|-|f(z)|>N-\frac{1}{N} -|f(z)|.
\]
Letting $N\to \infty$, the divergence property $\delta_{\mathcal{R}_\mathcal{D}}(z,f)=\infty $ holds for $m$-almost every $z\in [0,1]^2$.
\end{proof}
\section{Proof of Theorem \ref{main-2}}

\begin{proof}
By assumption, one can take $c\in (0,1)$ so small that 
\[ 
1\geq \max \left\{ \frac{\overline{a}}{a} ,\frac{a}{\underline{a}} \right\}  >c >0
\]
holds for every sufficiently small $a\in \mathcal{D}$. 
Suppose $\delta_{\mathcal{R}_\mathcal{C}}(z,f)=0$ for $m$-almost every $z\in \mathbb{R}^2$. 
For a given $\delta>0$ consider the set
\[
E_\delta =\left\{ z\in \mathbb{R}^2\colon \sup _{\substack{{\rm diam}A< \delta ,\\ z \in A \in \mathcal{R}_\mathcal{C}}}\left|\frac{1}{|A|} \int_{A} f\, dm-f(z)\right|<1\right\} .
\]
We have $|E_\delta|>0$ for $\delta >0$ small enough.
Then for the characteristic function $\mathbbm{1}_{E_\delta}$ of $E_\delta$, we will have by Theorem \ref{J-M-Z}, that almost all $z \in E_\delta$ are Lebesgue, namely for almost every $z \in E_\delta$ we have
\[
\lim_{\substack{{\rm diam}R\rightarrow 0, \\ z \in R \in \mathcal{R}}}\frac{|R\cap E_\delta|}{|R|}=1.
\]
Take $B\in \mathcal{R}_\mathcal{D}$ with $z \in B$ and ${\rm diam}B$ so small such that 
\begin{equation}\label{leb-pt}
\frac{|B\cap E_\delta|}{|B|}>1-c^2.
\end{equation}
We now wish to cover the rectangle $B$ with rectangles of sides $x,y\in \mathcal{C}$. 
More specifically, we have the following lemma. 

\begin{lemma} \label{covering}
There is a collection $\{ A_k\in \mathcal{R}_\mathcal{C}\}_{q=1}^n$ for some $n \leq \left( \frac{1}{c} +1\right) ^2$ such that
\begin{enumerate}
\item[{\rm (i)}] $B \subset \bigcup_{q=1}^n A_q$,
\item[{\rm (ii)}] $|B\cap A_q|\geq c^2|B|$ for every $q\in \{ 1,\dots ,n\}$,
\item[{\rm (iii)}] $\dfrac{|A_q|}{|B|}\leq \dfrac{1}{c^2}$ for every $q\in \{ 1,\dots ,n\}$.
\end{enumerate}
\end{lemma}
\begin{proof}
Let $B=[s, s + a]\times[t,t + b]$ for some small $a,b\in \mathcal{D}$ and some $(s,t)\in \mathbb{R}^2$. 
Then by assumption, there exist $x,y\in \mathcal{C}$ such that
\[
(1a)\quad 1\geq \frac{x}{a} >c \quad \text{or} \quad (2a)\quad 1\geq \frac{a}{x} >c, 
\]
and 
\[
(1b)\quad 1\geq \frac{y}{b} >c \quad \text{or} \quad (2b)\quad 1\geq \frac{b}{y} >c.
\]
Hence, consider the following four cases. 
For the case where $x\in \mathcal{C}$ satisfies $(2a)$ and $y\in \mathcal{C}$ does $(2b)$, by taking
\[
A_1=[s, s + x] \times [t, t + y], 
\]
we have $|B\cap A_1|=|B|$ and
\[
\frac{|A_1|}{|B|} =\frac{xy}{ab} <\frac{1}{c^2} .
\]

For the case where $x\in \mathcal{C}$ satisfies $(2a)$ and $y\in \mathcal{C}$ does $(1b)$, we consider
\[
A_\ell =\left[ s,s+x\right] \times \left[ t+(\ell -1)y,t+\ell y\right]
\]
for $\ell \in \{ 1,\dots ,\lfloor b/y\rfloor \} $, and also 
\[
A_{\lfloor b/y\rfloor +1}=\left[ s,s+x\right] \times \left[ t+b-y,t+b\right]
\] 
to cover up. 
Then we have 
\[
\frac{|B\cap A_\ell |}{|B|} =\frac{ay}{ab} >c,
\]
and 
\[
\frac{|A_\ell |}{|B|} =\frac{xy}{ab} \leq \frac{1}{c}
\]
for every $\ell \in \{ 1,\dots ,\lfloor b/y\rfloor +1\} $. 
The case where $x\in \mathcal{C}$ satisfies $(1a)$ and $y\in \mathcal{C}$ does $(2b)$ is similar, hence omit this case. 

For the case where $x\in \mathcal{C}$ satisfies $(1a)$ and $y\in \mathcal{C}$ does $(1b)$, we define
\[
A_{k,\ell }=\left[ s+(k-1)x,s+kx\right] \times \left[ t+(\ell -1)y,t+\ell y\right]
\]
for $k\in \{ 1,\dots ,\lfloor a/x\rfloor \} $ and $\ell \in \{ 1,\dots ,\lfloor b/y\rfloor \} $. 
Define also 
\[
A_{\lfloor a/x\rfloor +1,\ell }=\left[ s+a-x,s+a\right] \times \left[ t+(\ell -1)y,t+\ell y\right]
\]
for $\ell \in \{ 1,\dots ,\lfloor b/y\rfloor \} $, 
\[
A_{k,\lfloor b/y\rfloor +1}=\left[ s+(k-1)x,s+kx\right] \times \left[ t+b-y,t+b\right]
\]
for $k\in \{ 1,\dots ,\lfloor a/x\rfloor \} $, and 
\[
A_{\lfloor a/x\rfloor +1,\lfloor b/y\rfloor +1}=\left[ s+a-x,s+a\right] \times \left[ t+b-y,t+b\right]
\]
Denote the whole $\{ A_{k.\ell}\}$ by $\{ A_q\} _{q=1}^n$, where we have 
\[
n\leq \left( \left\lfloor \frac{a}{x} \right\rfloor +1\right) \left( \left\lfloor \frac{b}{y} \right\rfloor +1\right)\leq \left( \frac{1}{c} +1\right) ^2.
\]
We also have
\[
\frac{|B\cap A_q|}{|B|} =\frac{xy}{ab} >c^2,
\]
and 
\[
\frac{|A_q|}{|B|} =\frac{xy}{ab} \leq 1
\]
for every $q\in \{ 1,\dots ,n\} $. 

In all cases above, we have a collection $\{ A_q\in \mathcal{R}_\mathcal{C}\} _{q=1}^n$ such that
\[
B\subset \bigcup _{q=1}^nA_q.
\]
Lemma is obtained. 
\end{proof}

It follows from \eqref{leb-pt} and Lemma \ref{covering}-(iii) that $E_\delta \cap A_q\neq \emptyset$ for all $q\in \{ 1,\dots ,n\}$. 
Hence
\[
\left| \frac{1}{|A_q|} \int_{A_q}f\, dm\right| \leq 1 + f(z)
\]
for all $q\in \{ 1,\dots ,n\}$. 
Thus we have by Lemma \ref{covering} that 
\begin{align*}
\int_B f\, dm \leq \sum_{q=1}^{n}\int_{A_q}f\, dm \leq \left( 1 + f(z)\right) \sum_{q=1}^n|A_q| \leq \left( 1 + f(z)\right) \left( \frac{1}{c} +1\right) ^2\frac{1}{c^2} |B|,
\end{align*}
and which implies
\[
\frac{1}{|B|}\int_B f\, dm\leq \left( \frac{1}{c} +1\right) ^2\frac{1}{c^2} \left( 1 + f(z)\right).
\]
Consequently, we have
\[
\limsup _{\substack{{\rm diam} B\rightarrow 0, \\ z \in B \in \mathcal{R}_\mathcal{D}}}\frac{1}{|B|}\int_B f\, dm<\infty.
\]
Thus, for almost every Lebesgue points $z \in E_\delta$ we have that the upper differential is bounded. 
Then, due to a theorem of Besicovitch \cite{Besicovitch1935}, it follows that $\delta_{\mathcal{R}_\mathcal{D}}(z,f)=0$ for $m$-almost every $z\in \mathbb{R}^2$. 
\end{proof}

\section*{Acknowledgments} 

The first author was supported by Japan Society for the Promotion of Science (JSPS) KAKENHI Grant Number 19K03558. 
The second author is supported by the Knut and Alice Wallenberg foundation (KAW).

\end{document}